\documentclass[12pt]{amsart}
\usepackage{amscd}
\usepackage{amssymb,amsmath,amsthm}
\usepackage[all]{xy}
\usepackage{color}

\theoremstyle{plain}
\newtheorem{thm}{Theorem}[section]
\newtheorem{lem}[thm]{Lemma}

\newtheorem{cor}[thm]{Corollary}
\newtheorem{defn-lem}[thm]{Definition-Lemma}

\newtheorem{prop}[thm]{Proposition}
\newtheorem{question}[thm]{Question}

\theoremstyle{definition}
\newtheorem{defn}[thm]{Definition}

\newtheorem{rem}[thm]{Remark}

\def\md #1#2#3#4#5 {\left(
                        \begin{matrix}
             #1 & #2 \\
             #3 & #4
                        \end{matrix}
                      \right)- #5}

\def\ma #1#2#3#4 {\left(
                        \begin{matrix}
             #1 & #2 \\
             #3 & #4
                        \end{matrix}
                      \right)}
\def\Index{\operatorname{Index}}

\def \Ker {\operatorname{Ker}}

\def\QT{\operatorname{QT}}
\def\tsr{\operatorname{tsr}}

\def\Cu{\operatorname{Cu}}
\def\id{\operatorname{id}}
\def\T{\operatorname{T}}
\def\Aut{\operatorname{Aut}}

\newcommand{\mc}{\mathcal}
\begin{document}
\title [Tracially sequentially-split $\sp*$-homomorphisms between $C\sp*$-algebras II]
       { On permanence of regularity properties}

\begin{abstract}
We study a pair of $C\sp*$-algebras $(A, B)$ by associating a $*$-homomorphism from $A$ to $B$ which allows an approximate left inverse to the ultrapower $C\sp*$-algebra of $A$ as a completely positive map of order zero and show that important regularity properties related to the Elliott program pass from $B$ to $A$ in our setting.      
\end{abstract}

\author { Hyun Ho \, Lee and Hiroyuki Osaka}

\address {Department of Mathematics\\
          University of Ulsan\\
         Ulsan, 44610, South Korea 
 }
\email{hadamard@ulsan.ac.kr}

\address{Department of Mathematical Sciences\\
Ritsumeikan University\\
Kusatsu, Shiga 525-8577, Japan}
\email{osaka@se.ritsumei.ac.jp}


\date{}
\thanks{This first  author's research was supported by Basic Science Research Program through the National Research Foundation of Korea(NRF) funded by the Ministry of Education(2018R1D1A1B0705092414). The second author's was supported by the JSPS grant
for Scientific Research No. 20K03644.}
\maketitle

%%%%%%%%%%%%%%%%%%%%%%%%%%%%%%%%%%%%%%%%%%%%%%%%%%%%%%%%%%%%%
\section{Introduction}

There have been approaches to find a nice and large subalgebra of the crossed product $C\sp*$-algebra from which we can deduce a structure of the crossed product algebra especially related to a single automorphism case or under a $\mathbb{Z}$-action \cite{ABP, Phillips:Large}. However, a compact group action with  a certain Rokhlin  property recasts the relation between the original algebra and the crossed product $C\sp*$-algebra into a setting that there is a $*$-homomorphism from the crossed product algebra to a larger algebra which turns out to be isomorphic to the stabilization of the original algebra.   More precisely, when a finite(compact) group $G$ acts on a $C\sp*$-algebra $A$ and the automorphism $\alpha:G \curvearrowright A$ has the strict Rokhlin property,  there is a canonical map from $A\rtimes_{\alpha} G \to (C(G)\otimes A) \rtimes_{\sigma\otimes \alpha}G$, where $\sigma:G \curvearrowright C(G)$ is the translation action.  This is an important example to which  the following framework is applied:  It is said that a $*$-homomorphism $\phi:A \to B$ is sequentially split if there is a $*$-homomorphism $\psi:B \to A_{\infty}$,  which is called the approximate left inverse, such that $\psi(\phi(a))=a$ for all $a\in A$ \cite{BS}.  We note that this framework also works  for inclusions with the strict Rokhlin property though it has grown out of Toms and Winter's strongly self-absorbing $C\sp*$-algebras.\\  

 Since tracial Rokhlin properties for a finite group action have been developed in \cite{HO}, \cite{LO}, \cite{OT1}, and \cite{Phillips:tracial},  it is worthwhile to think of a tracial version of a sequentially-split $*$-homomorphism and we already suggested one in the name of tracially sequentially split $*$-homomorphism by replacing the rigid condition $\psi(\phi(a))=a$ with the condition $\psi(\phi(a))-a$ to be tracially small in a way reminiscent of H. Lin's tracial topological rank in \cite{LO1}.  In this article we further relax the condition for the approximate left inverse $\psi$ to be a $*$-homomorphism and extend our previous definition, i.e., we say that a $*$-homomorphism $\phi:A \to B$ is tracially sequentially-split by order zero map when the approximate left inverse $\psi$ is allowed to be an order zero map and $\psi(\phi(a))-a$ is tracially small for all $a\in A$. We note that this weakened definition requires more careful arguments based on the transition from projections to positive elements to prove the statements even appeared in \cite{BS} and believe that our work will shed light on systematic analysis of tracial approximation.  Moreover, while a tracially sequentially split $*$-homomorphism in \cite{LO1} provides the unified framework for Phillips' tracial Rokhlin property of a finite group action and inclusions with a tracial Rokhlin property due to Osaka and Teruya \cite{OT1}, our extended notion accommodates the generalized tracial Rokhlin property of a finite group action in \cite{HO} and inclusions with the generalized tracial Rokhlin property in \cite{LO}. This article is organized as follows: Sec. \ref{S:Pre} serves as a preliminary  in which we recall the definition of  a completely positive map of order zero and  Cuntz sub-equivalence for positive elements.  In addition, we  present related facts to be used throughout the article. In Sec. \ref{S:DEF},  we precisely describe when a $*$-homomorphism $\phi:A \to B$ is called  tracially sequentially-split by order zero map and  show, among other things,  that  $\mc{Z}$-stability and strict comparison passes from $B$ to $A$ if there exists such a map.   In Sec. \ref{S:Examples},  we prove that the generalized tracial Rokhlin property of a finite group action gives rise to a $*$-homomorphism that is tracially sequentially-split by order zero map and the same conclusion for the inclusions with the generalized tracial Rokhlin property and thereby unify the known permanence results which were obtained separately in the scattered literature.

%%%%%%%%%%%%%%%%%%%%%%%%%%%%%%%%%%%%%%%%%%%%%%%%%%%%%%%%%%%%%%%%%%%%%%%%%%%

\section{Preliminaries: Cuntz subequivalence and Order zero map}\label{S:Pre}

Let $A$ be a $C\sp*$-algebra. We write $A^{+}$ for the set of positive elements in $A$.  We denote by $M_{\infty}(A)$  the algebraic direct limit of the inductive system $(M_n(A), \phi_n)$ where 
 $\displaystyle \phi_n: M_n(A) \ni a \to \left( \begin{matrix} a &0 \\
 0 &0 \end{matrix} \right) \in M_{n+1}(A) $ and  $a \oplus b=  \left(\begin{matrix}
 a & 0 \\ 0 &b \end{matrix}\right) \in M_{n+m}(A)$ for $a\in M_n(A), b\in M_m(A)$. 
 
\begin{defn}(Cuntz)
For $a, b \in A^{+}$ we say that $a$ is \emph{Cuntz subequivalent} to $b$, written $a \lesssim b$ if there exists a sequence $(v_n)_{n=1}^{\infty}$ in $A$ such that $\displaystyle \lim_{n\to \infty}v_nbv^*_n=a$, and $a$ and $b$ are \emph{Cuntz equivalent} in $A$, written $a \sim b$ if $a \lesssim b$ and $b\lesssim a$. More generally, if $a$ is a positive element in $M_n(A)$ and if $b$ is a positive element in $M_m(A)$, then write $a \lesssim b$ if there exists a sequence of rectangular  matrices $\{ x_k\}_k$ in $M_{m,n}(A)$ with $x_k^*bx_k \to a$ using obvious matrix multiplication.
\end{defn}

  When we consider $\mathbb{K}\otimes A$ instead of $A$, then the same equivalence relation defines a semigroup $\Cu(A)=(\mathbb{K}\otimes A)^{+}/ \sim$ together with the commutative semigroup operation $\langle a \rangle + \langle b \rangle = \langle a \oplus b \rangle$ and the partial order $\langle a \rangle \le \langle b \rangle \iff a \lesssim b$. We also define the semigroup $W(A)= (M_{\infty}(A))^{+}/\sim$ with the same operation and order.\\ 

For $\epsilon >0$ let  $f_{\epsilon}$ be a function defined from $[0,\infty)$ to $[0,\infty)$ given by $\max\{t-\epsilon, 0 \}$. Then $(a-\epsilon)_{+}$ is defined as $f_{\epsilon}(a)$.\\

Note that $\lesssim$ is transitive and reflexive, and $\sim$ is an equivalence relation. We collect some facts in the following. 
\begin{lem}\label{L:Cuntz}
Let $A$ be a $C\sp*$-algebra. 
\begin{enumerate}
\item Let $a,b \in A^{+}$. Suppose $a \in \overline{bAb}$. Then $a \lesssim b$.
\item Let $a,b \in A^{+}$ be orthogonal (that is, $ab=0$ written $a \perp b$). Then $a+b \sim a \oplus b$.
\item Let $c \in A$. Then $cc^* \sim c^*c$.
\item Let $a, b \in A^{+}$. Then the following are equivalent. 
\begin{enumerate}
\item $a\lesssim b$, 
\item $(a-\epsilon)_{+} \lesssim b$ for all $\epsilon>0$,
\item for every $\epsilon > 0$ there is $\delta>0$ such that $(a-\epsilon)_{+} \lesssim (b-\delta)_{+}$, 
\item for every $\epsilon > 0$ there exist $\delta>0$ and $x\in A$ such that $(a-\epsilon)_{+}=x^*(b-\delta)_{+}x$. 
\end{enumerate}
\item Let $a_j \lesssim b_j$ for $j=1,2$. Then $a_1 \oplus a_2 \lesssim b_1\oplus b_2$. If also $b_1 \perp b_2$, then $a_1+a_2 \lesssim b_1+b_2$ (note that we do not require $a_1 \perp a_2$.)
\end{enumerate}
\end{lem}
\begin{proof}
Most of them can be found in Section 2 of \cite{KR} and the last one is proved in \cite[Proposition 2.5]{Ro:UHF2} and \cite[Proposition  1.1]{Cu}.
\end{proof}
\begin{lem}\cite[Lemma 2.4]{Phillips:Large}\label{L:orthogonal}
Let $A$ be a simple $C\sp*$-algebra which is not type I. Let $a \in A^{+} \setminus \{0\}$, and let $n$ be any positive integer. Then there exist nonzero positive elements $b_1,b_2, \dots, b_n \in A$ such that 
$b_1 \sim b_2 \cdots \sim b_n$  and $b_i \perp b_j $ for $i\ne j$, and such that $b_1+b_2+\cdots+b_n \in \overline{aAa}$.
\end{lem}
\begin{defn}\cite[Definition 1.3]{WZ}
Let $A$ and $B$ be $C\sp*$-algebras and let $\phi:A \to B$ be a completely positive map. It is said that $\phi$ has order zero  if for $a,b \in A^{+}$  $\phi(a)\phi(b)=\phi(b)\phi(a)=0$ whenever $ab=ba=0$.
\end{defn}
 We shall abbreviate a completely positive map of order zero as an order zero map. The following theorem is about the structure of an order zero map. 
\begin{thm}\cite[Theorem 3.3]{WZ}\label{T:orderzero}
Let $A$ and $B$ be $C\sp*$-algebras and $\phi:A \to B$ be a contractive order zero map. Let $C=C^*(\phi(A)) \subset B$. Then there is a positive element $\displaystyle h \in \mc{M}(C) \cap C^{\prime}$ with $\|h\|=\|\phi\|$ and a $*$-homomorphism 
 \[\pi_{\phi}:A \to \mc{M}(C)\cap \{h\}'\] such that for $a\in A$
\[ \phi(a)=h\pi_{\phi}(a).\]
If $A$ is unital, then $h=\phi(1_A) \in C$. 
\end{thm}
The following theorem is important since it enables us to lift an order zero map in some cases.  
\begin{thm} \cite[Corollary 4.1]{WZ}\label{T:semiprojective}
Let $A$ and $B$ be $C\sp*$-algebras, and $\phi:A \to B$ a contractive order zero map. Then the map given by $\varrho_{\phi} (\id_{(0,1]}\otimes a):=\phi(a)$ induces a $*$-homomorphism $\varrho_{\phi}:C_0((0,1])\otimes A \to B$. Conversely, any $*$-homomorphism $\varrho:C_0((0,1])\otimes A \to B$ induces a c.p.c. order zero map 
$\phi_{\varrho}(a):=\varrho(\id_{(0,1]}\otimes a)$.\\
These mutual  assignments yield a canonical correspondence between the space of c.p.c. order zero maps from $A$ to $B$  and the space of $*$-homomorphisms from $C_0((0,1])\otimes A $ to $B$.
\end{thm}
We finish this section with a simple observation. 
\begin{lem}\label{L:simpleorderzero}
Let $A$ be a simple $C\sp*$-algebra and $\phi:A \to B$ be a nonzero contractive order zero map. Then $\phi$ is injective. 
\end{lem}
\begin{proof}
Write $\phi(\cdot)=h\pi_{\phi}(\cdot)$ as in Theorem \ref{T:orderzero}. Consider $a\in A$ and $b\in \Ker \phi$. Then 
\[ \begin{split}
\phi(ab)&=h\pi_{\phi}(ab)\\
&=h\pi_{\phi}(a)\pi_{\phi}(b)\\
&=\pi_{\phi}(a)h\pi_{\phi}(b)\\
&=\pi_{\phi}(a)\phi(b)=0.
\end{split}
\]
So $ab \in \Ker \phi$. Similarly $ba \in \Ker \phi$. Thus $\Ker \phi$ is a closed ideal and it must be trivial since $A$ is simple.  
\end{proof}

%%%%%%%%%%%%%%%%%%%%%%%%%%%%%%%%%%%%%%%%%%%%%%%%%%%%%%%
\section{Tracially sequentially-split homomorphism between $C\sp*$-algebras and $\mc{Z}$-stability}\label{S:DEF}

Throughout the paper, we fix a free ultrafilter $\omega$ on $\mathbb{N}$ and recall that  
$l^{\infty}(\mathbb{N}, A)$ denotes  the $C\sp*$-algebra of  all bounded function from $\mathbb{N}$  to  $A$. We define a closed ideal of $l^{\infty}(\mathbb{N}, A)$ as follows: 
\[c_{\omega}(\mathbb{N}, A)=\{(a_n)_n \in l^{\infty}(\mathbb{N}, A)  \mid \lim_{\omega}\|a_n \|=0 \}. \]

Then we denote by $A_{\omega}=l^{\infty}(\mathbb{N}, A)/c_{\omega}(\mathbb{N}, A)$ the ultrapower $C\sp*$-algebra of $A$  with respect to the filter $\omega$ that is  equipped with the norm  $\|a\|= \lim_{\omega} \|a_n\| $ for $a=[(a_n)_n] \in A_{\omega}$. In addition, we denote by $\pi_{\omega}$ the canonical quotient map from $l^{\infty}(\mathbb{N}, A)$  onto $A_{\omega}$. Note that we can embed $A$ into $A_{\omega}$ as constant sequences, and we call $A_{\omega} \cap A'$ the central ultrapower algebra of $A$. For an automorphism of $\alpha$ on $A$, we also denote by $\alpha_{\omega}$ the induced automorphism on $A_{\omega}$ or $A_{\omega}\cap A'$ without confusion. 

Since $ \|a\|=\lim_{\omega}\|a_n\|$ for $a=[(a_n)_n] \in A_{\omega}$, for any nonzero projection $p=[(p_n)] \in A_{\omega}$ we may assume that each $p_n$ is a nonzero projection in $A$. Similarly, for any nonzero positive element $a \in A_{\omega}$ we can represent it as $a=[(a_n)_n]$ such that each $a_n$ is a nonzero positive element in $A$ and  $\{ \|a_n\| \mid n=1,2, \dots \}$ is uniformly away from zero. This is the main reason that we introduce $A_{\omega}$ rather than $A_{\infty}$.  When we consider $A_{\omega}$, we assume that $A$ itself is separable and unital. 

\begin{defn}\label{D:sequentiallysplitbyorderzero}
 A $^*$-homomorphism $\phi:A \to B$ is called tracially sequentially-split by order zero map, if for every positive nonzero element $z \in A_{\omega}$  there exist an order zero map $\psi: B \to A_{\omega}$ and a nonzero contractive positive element $g\in A_{\omega}\cap A'$ such that 
\begin{enumerate}
\item $\psi(\phi(a))=ag$ for all $a\in A$, 
\item $1_{A_{\omega}} -g \lesssim z$ in $A_{\omega}$.
\end{enumerate}
\end{defn}

 Although the diagram below is not commutative, we still use it to symbolize that $\phi$ is tracially sequentially-split by order zero map with its tracial approximate left inverse $\psi$;
\begin{equation*}\label{D:diagram}
\xymatrix{ A \ar[rd]_{\phi} \ar@{-->}[rr]^{\iota} && A_{\omega} &\\
                          & B \ar[ur]_{\psi: \,\text{order zero.}} }
\end{equation*}
\begin{rem}
\begin{enumerate}
\item When both $A$ and $B$ are unital, and $\phi$ is unit preserving, then $g=\psi(1_B)$. Moreover, if $g=1_{A_{\omega}}$, then $\phi$ is called a (strictly) sequentially split $*$-homomorphism by \cite{BS} since an order zero map $\psi$ is in this case a $*$-homomorphism. 
\item If $\phi:A\to B$ is a unital $*$-homomorphism and $g$ is a projection, then an order zero map $\psi$ corresponds to  a $*$-homomorphism. In this case, $\phi$ is called a tracially sequentially-split map with the consequence that $1-g$ is Murray-von Neumann equivalent to a projection in $\overline{zA_{\omega}z}$ \cite{LO1}. 
\end{enumerate}
\end{rem}
\begin{prop}\label{P:amplification}
 Let $A$ be a unital inifinite dimensional simple  $C\sp*$-algebra and $B$ be a unital $C\sp*$-algebra. Suppose that $\phi:A \to B$ is tracially sequentially-split by order zero map and unit preserving. Then its amplification $\phi\otimes \id_n: A\otimes M_n \to B\otimes M_n$ is tracially sequentially-split by order zero map for any $n\in \mathbb{N}$.  
\end{prop}
\begin{proof}
Let $z$ be a nonzero positive element in $(M_{n}(A))_{\omega} \cong M_n(A_{\omega})$. Write $z=(z_{ij})$ where $z_{ij} \in A_{\omega}$.  Note that there exists one $i_0$ such that $z_{i_0 i_0} \neq 0$. Without loss of generality we assume $i_0=1$. Then  
\[ 
z_{11}=(1\otimes E_{11}) z (1\otimes E_{11})  \lesssim z. 
\]
But for $z_{11}$ there exist positive elements $c_1, c_2, \dots, c_n \in A_{\omega}$ such that $c_ic_j=0$, $c_i \sim c_j$ for $i\ne j$, and such that $c_1+\cdots+c_n \lesssim z_{11}$ by Lemma \ref{L:orthogonal}. Now consider an order zero map $\psi$ such that 
\begin{enumerate}
\item[(i)] $\psi(\phi(a))=a\psi(1)$,
\item[(ii)] $1-\psi(1) \lesssim c_1$.
\end{enumerate} 
Then  by Lemma \ref{L:Cuntz}
\[ (1-\psi(1)) \oplus \cdots \oplus (1-\psi(1)) \lesssim c_1+\cdots +c_n \lesssim z_{11} \lesssim z. \]
Therefore 
\[1-(\psi\otimes \id_n)(1_{M_n(B)}) \lesssim z.\]
Moreover, 
\[\begin{split}
(\psi \otimes \id_n)(\phi \otimes \id_n)(a\otimes e_{ij})&=a \psi(1) \otimes e_{ij}\\
 &=(\psi \otimes \id_n )(1_{M_n(B)}) (a\otimes e_{ij}).
\end{split}
\]
Thus we showed that $\phi \otimes \id_n$ has a tracially approximate left inverse $\psi\otimes \id_n$ which is also an order zero map. 
\end{proof}

\begin{thm}\label{T:simple}
Let $A$ be a  unital infinite dimensional $C\sp*$-algebra and $B$ be a unital $C\sp*$-algebra. Suppose that  $\phi:A\to B$ is a unital injective $*$homomorphism which is tracially sequentially-split by order zero map.  If $B$ is simple, then $A$ is simple. Moreover, if $B$ is simple and stably finite, then so is $A$. 
\end{thm}
\begin{proof}
Let $I$ be a nonzero two sided closed ideal of $A$ and take a nonzero positive element  $x$ in $I$.  For $x$ in $I (\subset A \subset A_{\omega})$ we consider a tracially approximate left inverse $\psi:B \to A_{\omega}$ and a positive element $g \in A_{\omega}\cap A'$ such that $\psi(\phi(a))=ag$ for $a\in A$ and $1-g \lesssim x$.  Now we know $\phi(x)\neq 0$. Then there are elements $b_i$'s and $c_i$'s  in  $B$ such that $\sum_{i=1}^n b_i\phi(x)c_i=1_B$ since B is simple. Note that \[\begin{split}
 g&=\psi(1_B)=\psi(\sum_i b_i\phi(x)c_i)=g\pi_{\psi}(\sum_i b_i\phi(x)c_i)=\sum_i g \pi_{\psi}(b_i)\pi_{\psi}(\phi(x))\pi_{\psi}(c_i)\\
&=\sum_i \pi_{\psi}(b_i)g \pi_{\psi}(\phi(x))\pi_{\psi}(c_i)= \sum_i \pi_{\psi}(b_i)g x \pi_{\psi}(c_i)\\
&=\sum_i g^{1/4}\pi_{\psi}(b_i)g^{1/4}x g^{1/4}\pi_{\psi}(c_i)g^{1/4}.
\end{split}\]
Thus we obtained $b_i'=g^{1/4}\pi_{\psi}(b_i)g^{1/4}$ and $c_i'=g^{1/4}\pi_{\psi}(c_i)g^{1/4}$ such that 
 \[ \sum_i b_i' x c_i'=g. \]
On the other hand,  for any $\epsilon >0$ there is an element $r\in A_{\omega}$ such that $\|rxr^*-(1-g)\| < \epsilon$. Therefore we have
\begin{equation*}
\| rxr^*+\sum_i b_i'x c_i'-1\| =\| rxr^*- (1-g) + \sum_i b_i'xc_i'-g\| < \epsilon.
\end{equation*}
Thus  if we represent $r$ as $[(r_n)_n]$, $b_i'$ as $[(b'_{i,n})_n]$, and $c_i'$ as $[(c'_{i,n})_n]$ respectively, we have 
\[
\lim_{ \omega} \| r_nxr_n^* +\sum_i (b'_{i,n}) x(c'_{i,n}) -1_A\| < \epsilon, 
\]
 so there exists a set $X$ in $\omega$ such that for all $n \in X$
\[
 \| r_nxr_n^{*} +\sum_i b^{'}_{i,n} xc^{'}_{i,n} -1_A \| < 2\epsilon.
\]
This means that $1_A \in I$, therefore $A$ is simple.

 Next, suppose $B$ is simple and stably finite. By Proposition \ref{P:amplification}, it is enough to show that if $B$ is finite and simple, then $A$ is finite. Let $v$ be an isometry in $A$. Then consider $\phi(v)$ which is again isometry. But $B$ is finite so that $\phi(v)\phi(v^*)=1$. By applying $\psi$ to both sides, we get 
\[(vv^*-1)g=0.\] 
However, the map from $A$ to $A_{\omega}$  defined by  $x\mapsto xg$ is injective since $A$ is simple. It follows that $1-vv^*=0$, so we are done. 
\end{proof}

We turn to show that the important regularity properties pass from $B$ to $A$ provided that there is a $*$-homomorphism from $A$ to $B$ which is tracially sequentially-split by order zero map. Recall that the Jiang-Su algebra $\mc{Z}$ is a simple separable nuclear and infinite-dimensional $C\sp*$-algebra with a unique trace and the same Elliott invariant with $\mathbb{C}$ \cite{JS}. It is said that $A$ is $\mc{Z}$-stable or $\mc{Z}$-absorbing if $A\otimes \mc{Z}\cong A$. 

\begin{defn}(Hirshberg and Orovitz)\label{D:Z-absorbing}
A unital $C\sp*$-algebra $A$ is called tracially $\mc{Z}$-absorbing if $A \ncong \mathbb{C}$ and for any finite set $F\subset A$, $\epsilon >0$, and nonzero positive element $a\in A$ and $n\in \mathbb{N}$ there is an order zero contraction $\phi:M_n \to A$ such that the following hold:
\begin{enumerate}
\item $1-\phi(1)\lesssim a$, 
\item for any normalized element $x\in M_n$ and any $y\in F$ we have $\|[\phi(x), y]\| < \epsilon $.
\end{enumerate}
\end{defn}

\begin{thm}\label{T:Z-absorbing}
Let $A$ be a simple unital infinite dimensional $C\sp*$-algebra and $B$ be a unital $C\sp*$-algebra.  Suppose that $\phi:A\to B$ is a unital $*$-homomorphism which is tracially sequentially-split by order zero map.  If $B$ is tracially $\mc{Z}$-absorbing, then so is $A$. Thus, if $B$ is $\mc{Z}$-absorbing, then $A$ is also $\mc{Z}$-absorbing provided that $A$ is nuclear.    
\end{thm}
\begin{proof}
Let $F$ be a finite set of $A$, $\epsilon >0$, $n \in \mathbb{N}$, $z$ be a non-zero positive element in $A$.  There are mutually orthogonal positive nonzero elements $z_1, z_2$ in $\overline{zAz}$ such that $z_1+z_2 \lesssim z$.\\
Set $G=\phi(F)$ a finite set in $B$, then for $\phi(z_1)$ there is an order zero contraction  $\phi':M_n(\mathbb{C}) \to B$ such that 
\begin{enumerate}
\item $1-\phi'(1) \lesssim \phi(z_1)$,
\item $\forall x \in M_n(\mathbb{C})$ such that $\|x\|=1$, $\|[\phi'(x), y]\| < \epsilon$ for every $y\in G$.
\end{enumerate}  
For $z_2$ take a tracial approximate left inverse $\psi:B \to A_{\omega}$ for $\phi$ such that $1-\psi(1) \lesssim z_2$. 
Note that $\widetilde{\psi}:=\psi \circ \phi': M_n(\mathbb{C}) \to A_{\omega}$ is an order zero contraction. 
Then 
\begin{equation}\label{E:smallerror}
\begin{split}
1-\widetilde{\psi}(1)&=1-\psi(1) +\psi(1)-\psi(\phi'(1))\\
&=1-\psi(1)+\psi(1-\phi'(1))\\
&\lesssim z_2 +z_1\psi(1) \\
&\lesssim z_2+z_1 \lesssim z.
\end{split}
\end{equation}
Moreover, if we write $\psi=h\pi_{\psi}$,  for $a\in F$
\[
\begin{split}
[\widetilde{\psi}(x), a]&=\psi(\phi'(x))a-a\psi(\phi'(x))\\
&=h\pi_{\psi}(\phi'(x))a-ah\pi_{\psi}(\phi'(x)) \\
&=\pi_{\psi}(\phi'(x))ha-ah\pi_{\psi}(\phi'(x))\\
&=\pi_{\psi}(\phi'(x))\psi(\phi(a))-\psi(\phi(a))\pi_{\psi}(\phi'(x))\\
&=\pi_{\psi}(\phi'(x))h\pi_{\psi}(\phi(a))-h\pi_{\psi}(\phi(a))\pi_{\psi}(\phi'(x))\\
&=h \pi_{\psi}(\phi'(x)\phi(a)-\phi(a)\phi'(x))\\
&=\psi([\phi'(x), \phi(a)]).
\end{split}
\]
Therefore 
\begin{equation} \label{E:estimatefororderzeromap}
\|[\widetilde{\psi}(x), a]\| < \epsilon. 
\end{equation}
Since $C_0((0,1])\otimes M_n(\mathbb{C})$ is projective, there is an order zero lift $\widehat{\psi}: M_n(\mathbb{C}) \to l^{\infty}(\mathbb{N}, A)$ of $\widetilde{\psi}$  by Theorem \ref{T:semiprojective}.  Then we have the following diagram, 

\begin{equation}\label{D:diagram2}
\xymatrix{ &  l^{\infty}(\mathbb{N}, A)\ar[dd]_{\pi_{\omega}}\\ 
 &  \\
  M_n(\mathbb{C}) \ar[ruu]^{\widehat{\psi}} \ar[r]_{\widetilde{\psi}} & A_{\omega} \,. }
\end{equation}
Thus we can write $ \widehat{\psi}(x)=(\widehat{\psi}_k(x))$, where each $\widehat{\psi}_k$ is an order zero map from $M_n(\mathbb{C})$ to $A$  and let $h=[(h_k)_k]$ where $\widehat{\psi}_k(1_{M_n})=h_k$. 

Note that (\ref{E:smallerror}) implies that for  any $\epsilon>0$ there exist $\delta>0$ and $s\in A_{\omega}$ such that 
\[(1-\widetilde{\psi}(1)-\epsilon)_{+}=s(z-\delta)_{+}s^*.\] 
Write $s=[(s_n)_n]$. Then from (\ref{D:diagram2})
\[  \lim_{\omega} \| (1- h_k-\epsilon)_+ -(s_k(z-\delta)_{+}s_k^*) \|=0.\]
It follows that there exists a set $X$ in $\omega$ such that for all $k\in X$ 
\[ \|(1- h_k-\epsilon)_+ - s_k(z-\delta)_{+}s_k^*\| < \epsilon,\]
which implies that 
\[(1-\widehat{\psi}_k(1_{M_n}) -2\epsilon)_{+} \lesssim (z-\delta)_{+}.\] 
Therefore, by Lemma \ref{L:Cuntz} 
\[  1-\widehat{\psi}_k(1_{M_n})  \lesssim z. \]
In addition,  from (\ref{E:estimatefororderzeromap}) 
\begin{equation*}
\| [\widehat{\psi}_k(x), a]\| < 2\epsilon \quad \text {for all normalized $x \in M_n$ and $a \in F$}
\end{equation*}
for such $k$'s in $X$. Hence we showed the existence of an order zero map from $M_n$ to $A$ satisfying the conditions in Definition \ref{D:Z-absorbing}. The last statement follows from \cite[Theorem 4.1]{HO} which is essentially a part of \cite{MS}. 
\end{proof}
\begin{cor}\label{C:tsr}
Let $A$ and $B$ be unital nuclear $C\sp*$-algebras. Suppose that there is a $*$-homomorphism from $A$ to $B$ which is tracially sequentially-split by order zero map. Then if $B$ is simple, (stably) finite, and $\mc{Z}$-absorbing, then $A$ has stable rank one.   
\end{cor}
\begin{proof}
Note that the stable rank of $B$ is one by \cite[Theorem 6.7]{Ro}. By Theorem \ref{T:simple}, \ref{T:Z-absorbing}, $A$ is simple and (stably) finite, and absorbs the Jiang-Su algebra $\mathcal{Z}$. Then again by  \cite[Theorem 6.7]{Ro} $\tsr(A)=1$. 
\end{proof}

 Next we turn to strict comparison of positive elements.    We denote by  $\QT(A)$ the space of normalized 2-quasitraces on $A$. Given $\tau \in \QT(A)$, we define a lower semicontinuous map $d_{\tau}:M_{\infty}(A)^{+} \to \mathbb{R}^{+}$ by 
\[d_{\tau}(a)=\lim_{n \to \infty} \tau(a^{1/n})\]  where $M_{\infty}(A)^{+}$ denotes the positive elements in $M_{\infty}(A)$. 
If  $a\lesssim b$ whenever $d_{\tau}(a) < d_{\tau}(b)$ for every $\tau \in \QT(A)$, then we say that $A$ has \emph{strict comparison of positive elements} or shortly \emph{strict comparison} \cite{ET}.  If we let $\T(A)$ be the space of tracial states of $A$,  $\QT(A)=\T(A)$ when $A$ is exact by \cite{Ha}.

\begin{defn} \cite[Definition 3.1]{Phillips:Large}
Let $A$ be a $C\sp*$-algebra and $a \in (\mathbb{K}\otimes A)^{+}$ is called purely positive if $a$ is not Cuntz equivalent to a projection in  $(\mathbb{K}\otimes A)^{+} $. We denote $\Cu_{+}(A)$
by the set of elements  $\eta \in \Cu(A)$ which are not the classes  of projections, and similarly $W_{+}(A)$ by the set of elements $\eta \in W(A)$ which are not the classes of projections. 
\end{defn}

\begin{prop}\label{P:cuntz}
Let $A, B$ be unital, separable, simple $C\sp*$-algebras and assume that there is a unital $*$-homomorphism $\phi:A \to B$ which is tracially sequentially-split by order zero map. If $\phi(a) \lesssim \phi(b)$ for two positive elements $a, b \in A$ with $b$ being purely positive, then $a\lesssim b$.
\end{prop}
\begin{proof}
The role of a tracial approximate left inverse for $\phi$ is not changed under the weaker assumption that it is an order zero map rather than a  $*$-homomorphism since an order zero map still preserves the relation $\lesssim$. See the proof of \cite[Proposition 2.18]{LO1}.   
\end{proof}

\begin{thm}\label{T:strictcomparison}
Let $A$ be a unital, stably finite, simple, exact,  separable,  infinite dimensional $C\sp*$-algebra which is not type I and $B$ a unital, stably finite, simple, exact, separable,  infinite dimensional $C\sp*$-algebra. Suppose that there is a unital $*$-homomorphism $\phi:A \to B$ which is tracially sequentially-split by order zero map.   If $B$ has strict comparison, so does $A$.   
\end{thm}
\begin{proof}
Again the proof is almost same as one in \cite[Theorem 2.22]{LO1} with a care that the composition of a trace and an order zero map is again a trace  by \cite[Corollary 4.4]{WZ}.   Moreover,  a tracial approximate left inverse for $\phi$ is needed in Proposition \ref{P:cuntz} only and \cite[Lemma 3.6]{Phillips:Large} plays a critical role to control elements in $A$ before applying Proposition \ref{P:cuntz}. 
\end{proof}

%%%%%%%%%%%%%%%%%%%%%%%%%%%%%%%%%%%%%%%%%%%%%%%%%%%%%%%%%%%%%%%%%%%%%%%%%%%%%%%%%%%%
\section{Applications}\label{S:Examples}

In this section, all groups $G$ are assumed to be discrete and its action on a $C\sp*$-algebra  $A$ given by the map $\alpha:G \to \Aut(A)$ is denoted by $\alpha: G \curvearrowright A$.  
\begin{defn}
Let $A$ and $B$ be unital $C\sp*$-algebras. Given two actions  $\alpha:G \curvearrowright A$, $\beta:G \curvearrowright B$, an equivariant $*$-homomorphism $\phi:(A, \alpha) \to (B, \beta) $ is called  $G$-tracially sequentially-split by order zero map, if for every nonzero positive element $z\in A_{\omega}$ there exists an equivariant tracial approximate left inverse $\psi:(B, \beta) \to (A_{\omega},\alpha_{\omega})$ which has order zero. 
\end{defn}

\begin{defn}(Hirshberg and Orovitz \cite{HO})
Let $G$ be a finite group and $A$ be a separable unital $C\sp*$-algebra. We say that $\alpha: G \curvearrowright A$ has the generalized tracial Rokhlin property if  for every finite set $F \subset A$, every $\epsilon>0$, any nonzero positive element $x\in A$ there exist normalized positive contractions $\{e_g\}_{g\in G}$ such that 
\begin{enumerate}
\item $e_g \perp e_h$ when $g\ne h$,
\item $\| \alpha_g(e_h)-e_{gh} \| \le \epsilon$, \quad for all $g, h \in G$,
\item $\|  e_gy -ye_g \| \le \epsilon$, \quad for all  $g \in G$, $ y \in F$,
\item  $1-\sum_{g \in G}e_g \lesssim x$. 
\end{enumerate}
\end{defn}

In virtue of $A_{\omega}$ we can express the generalized tracial Rokhlin property of $\alpha: G \curvearrowright A$ with exact relations as follows.  
\begin{thm}\label{T:tracialRokhlinaction}
Let $G$ be a finite group and $A$ be a separable unital $C\sp*$-algebra. Suppose that  $\alpha:G \curvearrowright A$ has the generalized tracial Rokhlin property. Then  for any nonzero positive element $x \in A_\omega$ there exist mutually orthogonal  positive contractions  $e_g$'s in $A_{\omega}\cap A'$ such that 
\begin{enumerate}
\item $\alpha_{\omega, g}(e_h)=e_{gh}$ for all $g,h\in G$, where $\alpha_{\omega}:G \curvearrowright A_{\omega} $ is the induced action,  
\item 
$1-\sum_{g \in G} e_g \lesssim x$.
\end{enumerate}
\end{thm}
\begin{proof}
The proof is almost same as the first part in the proof of \cite[Theorem 3.3]{LO1}; when we write $x=[(x_n)_n]$ where $x_n$'s are nonzero positive elements  in $A$, we can construct a sequence $\{e_{g,n}\}$ of positive contractions such that $1-\sum_{g\in G} e_{g,n} \lesssim x_n$. Then for given $\epsilon$ we can construct $r \in A_{\omega}$ such that $\| r x r^*-(1- \sum_{g\in G} e_g)    \| < \epsilon$
where $e_g=[(e_{g,n})_n]$. This implies that $1-\sum_{g \in G} e_g \lesssim x$.
\end{proof}

Let $C(G)$ be the algebra of complex valued continuous functions on $G$  and  $\sigma:G \curvearrowright C(G)$ the canonical translation action. 

\begin{thm}\label{T:tracialRokhlinaction2}
Let $G$ be a finite group and $A$ a separable unital $C\sp*$-algebra. Suppose that $\alpha:G \curvearrowright A$ has the generalized tracial Rokhlin property. Then for every nonzero positive element $x$ in $A_{\omega}$ there exists a $*$-equivariant order zero map $\phi$ from $(C(G), \sigma)$ to $(A_{\omega}\cap A', \alpha_{\omega})$ such that $1-\phi(1_{C(G)}) \lesssim x$ in $A_{\omega}$.  
\end{thm}

\begin{proof}
By Theorem \ref{T:tracialRokhlinaction}, for any nonzero positive $x\in A_{\omega}$ we can take mutually orthogonal positive contractions $e_g$'s in $A_{\omega}\cap A'$ such that $1-\sum e_g \precsim x$. Then we define $\phi(f)=\sum_g f(g)e_g$ for $f\in C(G)$. It follows that it is an order zero map and $1-\phi(1_{C(G)})=1-\sum_g e_g \precsim x$ .  Using the condition (1) in Theorem \ref{T:tracialRokhlinaction}, it is easily shown that $\phi$ is equivariant.
\end{proof}

\begin{cor}\label{C:leftinverse}
Let $G$ and $A$ be as same as Theorem \ref{T:tracialRokhlinaction2}. Suppose that  $\alpha:G \curvearrowright A$ has the generalized tracial Rokhlin property. Then the map $1_{C(G)}\otimes \id_A : (A, \alpha) \to (C(G)\otimes A, \sigma \otimes \alpha)$ is $G$-tracially sequentially-split by order zero map. 
\end{cor}
\begin{proof}
Consider a $*$-homomorphism $m: ((A_{\omega}\cap A')\otimes A, \alpha_{\omega}\otimes \alpha) \to (A_{\omega}, \alpha_{\omega})$ defined by sending $\mathbf{a}\otimes x$ to $\mathbf{a}x$ (see \cite{BS}).  It is easily checked that $m$ is equivariant. Then for any nozero  positive $z \in A_{\omega}$ consider a $*$-equivariant order zero map $\phi$ from $(C(G), \sigma)$ to $(A_{\omega}\cap A', \alpha_{\omega})$ such that $1- \phi(1_{C(G)}) \lesssim z$ in $A_{\omega}$ as in Theorem \ref{T:tracialRokhlinaction2}. Then  $\psi:=m \circ (\phi\otimes \id_A)$ is a $*$-equivariant order zero map  such that the following diagram holds.

\[ \xymatrix{ (A,\alpha) \ar[rd]_{1_{C(G)}\otimes \id_A} \ar@{-->}[rr]^{\iota} && ((A_{\omega}\cap A' )\otimes A, \alpha_{\omega} \otimes \alpha) \ar[r]^-{m} &(A_{\omega}, \alpha_{\omega}) \\
                          & (C(G)\otimes A, \sigma\otimes \alpha)\, . \ar[ur]^{\phi\otimes \id_A} \ar[urr]_{\psi} } \]  
In other words, $\psi$ is a $*$-equivariant traical approximate left inverse for $1_{C(G)}\otimes \id_A$.  
\end{proof}

As usual, $A\rtimes_{\alpha} G$ denotes the crossed product $C\sp*$-algebra of $(A, \alpha:G \curvearrowright A)$.  Moreover, we let $u:G \to U(A\rtimes_{\alpha}G)$ be the implementing (universal) unitary representation for $\alpha$. Thus  we write an element of $A\rtimes_{\alpha}G$ as $\sum_{g\in G}a_gu_g$  so that  $u_g a_h u^*_g=\alpha_g(a_h)$ for $g, h \in G$. The embedding of $A$ into $A\rtimes_{\alpha}G$ is given by the map $a \mapsto au_e$, where $e$ is the identity element of $G$. 

\begin{prop}
We denote by $\phi\rtimes G$ a natural extension of an equivariant map $\phi:(A,\alpha) \to (B,\beta)$  from $A\rtimes_{\alpha} G$ to $B\rtimes_{\beta} G$, where $\alpha:G \curvearrowright A$ and $\beta:G \curvearrowright B$.  Then $\phi \rtimes G$ is an order zero map if $\phi$ is an order zero map.  
\end{prop}
\begin{proof}
Note that $(\phi \rtimes G) (\sum_{g \in G}a_g u_g)= \sum_{g} \phi(a_g) u_g$. Let  $a=\sum_{g}a_gu_g$ and $ b=\sum_h b_hu_h$  be orthogonal in $A\rtimes_{\alpha} G$.  Then 
\[  ab=\sum_{g, h \in G} a_g \alpha_g(b_h) u_{gh}=0.\] By taking $gh=t$,  we have
\[ \sum_t \left(\sum_g a_g \alpha_g(b_{g^{-1}t}) \right)u_t=0.\]
Write $\phi(\bullet)= y \pi (\bullet)$, where $\phi(1_A)=y \in \pi(A)^{`}$. 
Since $(\phi \rtimes G)( \sum_t \left(\sum_g a_g \alpha_g(b_{g^{-1}t}) \right)u_t)=0$,  it follows that 
\[ \sum_t \phi(\sum_g a_g \alpha_g(b_{g^{-1}t}))u_t=0.\]
Therefore,  using $y(\sum_g c_g u_g)=\sum_g y c_g u_g$ twice,

\begin{align*}
0 &= \sum_t y \left(\sum_g  \phi(a_g \alpha_g(b_{g^{-1}t}))\right) u_t =\sum_t \left(\sum_g y \phi(a_g \alpha_g(b_{g^{-1}t}))\right) u_t \\
&= \sum_t \left(\sum_g y^2 \pi(a_g) \pi(\alpha_g(b_{g^{-1}t}))\right) u_t=\sum_t \left(\sum_g y \pi(a_g) y\pi(\alpha_g(b_{g^{-1}t}))\right) u_t\\
&= \sum_t \left( \sum_g \phi(a_g) \phi(\alpha_g(b_{g^{-1}t}))\right) u_t =\sum_t \left(\sum_g \phi(a_g)  \alpha_g (\phi(b_{g^{-1}t}))\right) u_t\\ 
&= \sum_h \left(\sum_g \phi(a_g)  \alpha_g (\phi(b_{h}))\right) u_{gh} =( \sum_g \phi(a_g)u_g)(\sum_h\phi(b_h)u_h). 
\end{align*}
Thus we have shown that $(\phi \rtimes G)(a) (\phi \rtimes G)(b)=0$. 

Since $M_n(A \rtimes_{\alpha} G) \cong M_n(A) \rtimes_{\alpha}G$ for a finite group $G$, to show that $\phi\rtimes G$ is completely positive it is enough to show that $(\phi \rtimes G)(c)$ is positive whenever $c$ is positive in $A\rtimes_{\alpha}G$. So consider $a^*a \ge 0$ where $a=\sum_g a_gu_g$ in $A \rtimes_{\alpha}G$.  
\[
\begin{split}
a^*a &= \left( \sum_g a_g u_g\right)^* \left(\sum_g a_g u_g \right) \\
&=\sum_{g,h} u_{g^{-1}}a^*_ga_h u_h = \sum_{g,h} \alpha_{g^{-1}}(a_g^* a_h)u_{g^{-1}h}.
\end{split}
\]  
Therefore, 
\[
\begin{split}
(\phi \rtimes G)(a^*a) &=  \sum_{g,h} \phi(\alpha_{g^{-1}}(a_g^* a_h))u_{g^{-1}h} =\sum_{g,h}\alpha_{g^{-1}}(\phi(a^*_g a_h))u_{g^{-1}h}\\
&=\sum_{g,h}\alpha_{g^{-1}}(y^{1/2}\pi(a^*_g)y^{1/2}\pi(a_h))u_{g^{-1}h}=\sum_{g,h} u_{g^{-1}}(y^{1/2}\pi(a_g))^*(y^{1/2}\pi(a_h))u_h\\
&=\left(\sum_g y^{1/2}\pi(a_g)u_g \right)^*\left( \sum_h y^{1/2}\pi(a_h)u_h\right) \ge 0.
\end{split}
\]
Thus we have shown that $\phi \rtimes G$ is completely positive too.   
\end{proof}

\begin{lem}\cite[Lemma 5.1]{HO}\label{L:positiveinA}
Let $A$ be a unital, simple, infinite dimensional  $C\sp*$-algebra and $\alpha:G\curvearrowright A$  an action of a finite group $G$ on $A$ such that $\alpha_g$ is outer for all $g\in G \setminus \{e\}$. Then for every nonzero positive element  $z \in A\rtimes_{\alpha} G$  there exists a nonzero positive element $x\in A$ such that $x \lesssim z$.   
\end{lem}

\begin{thm}\label{T:Actioncase}
Let $G$ be a finite group and $A$ a unital, simple,  infinite dimensional $C\sp*$-algebra. Suppose that $\alpha:G \curvearrowright A$ has the generalized tracial Rokhlin property. Then the $*$-homomorphism $ (1_{C(G)} \otimes id_A) \rtimes  G$ from $A\rtimes_{\alpha} G$ to $(C(G)\otimes A)\rtimes_{\sigma\otimes\alpha} G$ is tracially sequentially-split by order zero map. 
\end{thm}
\begin{proof}
Take a nonzero positive element $z$ in $(A\rtimes_{\alpha}G)_{\omega} $.  Since $\alpha:G\curvearrowright A$ is outer  by \cite[Proposition 5.3]{HO},  Lemma \ref{L:positiveinA} implies that we can have a nonzero positive element $x$ in $A_{\omega}$ such that $x \lesssim z$. By Theorem \ref{T:tracialRokhlinaction2} there is an equivarinat order zero map $\phi:C(G) \to A_{\omega}$ such that  $1_{A_{\omega}}-\phi(1_{C(G)}) \lesssim x \lesssim z$.  Note that  $ ( \psi \rtimes G) \circ (1_{C(G)}\otimes \id_A)\rtimes G)(1_A u_e) = [\phi(1_{C(G)})]u_e$ where $\psi$ is as in Corollary \ref{C:leftinverse}. 
Moreover, $1_{A_{\omega}}u_e - \phi(1_{C(G)})u_e \lesssim z$. It follows that  $(\psi \rtimes G)$ is a tracial approximate left inverse for $(1_{C(G)} \otimes \id_A)\rtimes G$. Here we implicitly use the natural embedding from $A_{\omega} \rtimes_{\alpha_{\omega}} G$ to $(A\rtimes_{\alpha} G)_{\omega}$. 
\[ 
\xymatrix{ (A\rtimes_{\alpha}G) \ar[rd]_{(1_{C(G)} \otimes \id_A)\rtimes G} \ar@{-->}[rr]^{\iota\rtimes G} && (A_{\omega}\rtimes_{ \alpha_{\omega}} G) \to (A\rtimes_{\alpha} G)_{\omega} &\\
                          & (C(G)\otimes A)\rtimes_{\sigma\otimes \alpha}G \, .\ar[ur]_{\psi \rtimes G} } \]  
\end{proof}

\begin{cor}\label{C:permanencebyRokhlinaction}
Let $G$ be a finite group and $A$ be a unital, separable, infinite dimensional $C\sp*$-algebra. Suppose that $\alpha:G \curvearrowright A$ has the genralized tracial Rokhlin property. Then if $A$ has the following properties, then so does $A\rtimes_{\alpha} G$.  
\begin{enumerate}
\item simple, 
\item simple and $\mc{Z}$-absorbing provided that $A$ is nuclear,
\item simple and strict comparison property provided that $A$ is exact and is not type I,
\item simple and stably finite.
\end{enumerate}
\end{cor}
\begin{proof}
Since 
\[ \begin{split}
 (C(G)\otimes A) \rtimes_{\sigma\otimes \alpha} G &\simeq (C(G)\rtimes_{\sigma}G)\otimes A\\
&\simeq M_{|G|}(\mathbb{C})\otimes A, 
\end{split}\] 
$(C(G)\otimes A) \rtimes_{\sigma\otimes \alpha} G$ does share the same structural property with $A$. Moreover,  the $*$-homomorphism $(1_{C(G)} \otimes id_A) \rtimes  G: A\rtimes_{\alpha} G \to (C(G)\otimes A)\rtimes_{\sigma\otimes\alpha} G \simeq M_{|G|}(A)$ is tracially sequentially-split by order zero map  by Theorem \ref{T:Actioncase}. Therefore the conclusions follow from Theorem \ref{T:simple}, Theorem \ref{T:Z-absorbing}, Theorem \ref{T:strictcomparison}.
\end{proof}

\begin{cor}
Let $G$ be a finite group, $A$ be a unital, separable, finite, infinite dimensional, simple, nuclear  $C\sp*$-algebra and  $\alpha:G \curvearrowright A$ have the generalized tracial Rokhlin property.  Assume that $A$ absorbs the Jiang-Su algebra $\mathcal{Z}$. Then $\tsr(A\rtimes_{\alpha}G)=1$. 
\end{cor}
\begin{proof}
It follows from Theorem \ref{T:Actioncase} and Corollary \ref{C:tsr}.
\end{proof}%-----------------------------------------------------------------------------------------------------------------------------------------
Another important example of a $*$-homomorphism which is tracially sequentially-split by order zero map  is provided by an inclusion of unital $C\sp*$-algebras of index-finite type with the generalized tracial Rokhlin property.  
\begin{defn}[Watatani]
Let $P\subset A$ be an inclusion of unital $C\sp*$-algebras and $E:A \to P$ a conditional expectation. Then we say that $E$ has a quasi-basis if there exist elements $\{(u_k,v_k)\}$ for $k=1,\dots, n$ such that for any $x\in A$
\[ x=\sum_{j=1}^nu_jE(v_jx)=\sum_{j=1}^n E(xu_j)v_j.\]
In this case, we define the Watatani index  of $E$ as 
\[\Index E= \sum_{j=1}^n u_jv_j.\] In other words, we say that $E$ has a finite index if there exists a quasi-basis. 
\end{defn}

Let $P\subset A$ be an inclusion of unital $C\sp*$-algebras and $E:A \to P$ a conditional expectation.
Unless stated otherwise  we assume that $E$ is faithful. Let $\mc{E}_{E}$ be the Hilbert $P$-module completion of $A$ by the norm given by a $P$-valued Hermitian bilinear form $\langle x, y \rangle_P=E(x^*y)$ for $x,y \in A$. As usual $\mc{L}(\mc{E}_E)$ denotes the algebra of  adjointable bounded operators on $\mc{E}_E$. There are an injective $*$-homomorphism $\lambda:A \to \mc{L}(\mc{E}_E)$ defined by a left multiplication and the natural inclusion map $\eta_E$ from $A$ to $\mc{E}_E$. 
Then the \emph{Jones projection} $e_P$ is defined by 
\[ e_P(\eta_E(x))=\eta_E(E(x)).\] 
Note that by \cite[Lemma 2.11]{W:index}, 
\begin{equation}
e_P x e_P= E(x)e_P
\end{equation}
for any $x\in A$.  For more details and related properties of an inclusion of unital $C\sp*$-algebras of index-finite type  we refer the reader to \cite{OKT:Rokhlin}, \cite{OT1}, \cite{W:index}. 

\begin{defn}\cite[Definition 3.2]{LO}
Let $P\subset A$ be an inclusion of unital $C\sp*$-algebras such that a conditional expectation $E:A\to P$ has a finite index and $E_{\omega}$ be the induced map from $A_{\omega}$ to $P_{\omega}$. It is said that $E$ has the generalized tracial Rohklin property if for every nonzero positive element $z\in A_{\omega}$ there is a nonzero positivie contraction $e\in A_{\omega}\cap A'$ such that 
\begin{enumerate}
\item $(\Index E)e^{1/2}e_Pe^{1/2}=e$,
\item $1-(\Index E)E_{\omega}(e) \lesssim z$ in $A_{\omega}$, 
\item $A\ni x \to xe \in A_{\omega}$ is injective. 
\end{enumerate} 
We call $e$ satisfying (1) and (2) a Rokhlin contraction. 
\end{defn}

As we notice, the third condition is automatically satisfied when $A$ is simple.  A typical example of  an inclusion of unital $C\sp*$-algebras of index-finite type arises from a finite group action $\alpha$ of $G$ on a unital $C\sp*$-algebra $A$; let $A^{\alpha}$ be the fixed point algebra, then the conditional expectation 
\[E(a)=\frac{1}{|G|} \sum_{g\in G} \alpha_g(a) \]   is of index-finite type if the action $\alpha:G\curvearrowright A$ is free \cite{W:index}. Moreover,  the following observation was obtained by us. 

\begin{prop}\cite[Proposition 3.8]{LO}
Let $G$ be a finite group, $\alpha$ an action of $G$ on a finite, infinite dimensional,  simple,  separable,  unital $C^*$-algebra $A$, and E as above.  Then $\alpha$ has the generalized tracial Rokhlin property if and only if $E$ has the generalized tracial Rokhlin property.
\end{prop}

We would like to show that if an inclusion $P\subset A$ of index-finite type has the generalized tracial Rokhlin property,  the embedding $P \hookrightarrow A$ is tracially sequentially-split by order zero map. A serious issue is to transfer $\lesssim$ in $A$ to $\lesssim$ in $P$. Though the following technical lemmas appeared in \cite{LO}, we rephrase them here for the convenience of the reader.    

\begin{lem}\label{L:technical}
Let $P \subset A$ be an inclusion of unital $C\sp*$-algebras of index-finite type. Suppose that $p,q$ are nonzero positive elements in $P_{\omega}$ such that $q \lesssim e^2 p$ in $A_{\omega}$ and $pe=ep$ where $e$ is a nonzero positive contraction in $A_{\omega}\cap A'$ such that $(\Index E)e^{1/2}e_Pe^{1/2}=e$. Then $q \lesssim p$ in $P_{\omega}$. 
\end{lem}
\begin{proof}
This is an easy modification of \cite[Lemma 3.12]{LO1}. See \cite[Lemma 5.2]{LO} for details.
\end{proof} 

\begin{lem}\label{L:normcontrol}
Let $A$ be a unital, finite,  simple,  infinite dimensional $C\sp*$-algebra and $x $ a nonzero positive element of $A$ such that $\| x\|=1$. For any $\epsilon >0$ there exists $y$ in $\overline{xAx} \setminus \{0\}$ such that whenever $g\in A_{\omega}^{+}$ satisfies $0\le g \le 1$ and $g \lesssim y$ in $A_{\omega}$, then $\| (1-g)x(1-g)\| > 1- \epsilon$.
\end{lem}
\begin{proof}
It essentially follows from \cite[Lemma 2.6]{Phillips:Large} since any nonzero positive element in $A_{\omega}$ can be represented by a sequence of nonzero positive elements in $A$.  
\end{proof}

\begin{thm}\label{T:Inclusion}
Let $P\subset A$ be an inclusion of unital, infinite dimensional, separable, $C\sp*$-algebras of index-finite type where $A$ is simple and  finite. Suppose that a conditional expectation $E:A \to P$  has the generalized tracial Rokhlin property. Then the embedding $\iota: P \hookrightarrow A$ is  tracially sequentially-split by order zero map.  
\end{thm}
\begin{proof}
Let $z$ be a nonzero positive element $P_{\omega} \subset\ A_{\omega}$.   Using the idea in \cite[Theorem 3.13]{LO1} or essentially the ``diagonal sequence argument"  we can construct a nonzero positive contraction $f$ such that $fz=zf$ and  $(\Index E)f^{1/2}e_Pf^{1/2}=f$; write  $z=[(z_k)_k]$ where each $z_k$ is a nonzero positive element in $P$. For $\displaystyle \epsilon = 1-\frac{(\Index E)^2}{4} $ there exists a positive element $y_k$ in $\overline{z_k^2 Az_k^2}\setminus \{0\}$ such that whenever $g \in A_{\omega}^{+}$, $0\le g \le 1$, and $g\lesssim y_k$ in $A_{\omega}$, then $\|(1-g)(z_k)^2(1-g)\| > 1-\epsilon$ by Lemma \ref{L:normcontrol}. Then for $y_k$ consider a Rokhlin contraction $e_k \in A_{\omega}\cap A'$ such that $(\Index E)(e_k) e_k^{1/2}e_Pe_k^{1/2}=e_k$  in $A_{\omega}$ and $1-(\Index E)E_{\omega}(e_k) \lesssim y_k$ in $A_{\omega}$.  By putting $g=1-(\Index E)E_{\omega}(e_k)$ we obtain that   
\begin{equation}\label{E:nonzero}
\| (\Index E)E_{\omega}(e_k)(z_k)^2(\Index E) E_{\omega}(e_k) \| > \dfrac{(\Index E)^2}{4}.
\end{equation}
 Write $e_k=[(e^k_n)_n]$. Since $E$ is contractive, we have  
\begin{enumerate}
\item $\displaystyle \lim_{n \to \omega}\|(\Index E)(e^k_n)^{1/2} e_P (e^k_n)^{1/2} -e^k_n\|=0$,
\item $\displaystyle \lim_{n \to \omega}\| e^k_n z_k-z_ke^k_n\|=0$,
\item $\displaystyle \lim_{n \to \omega}\|e^k_na-ae^k_n\|=0$ for $a\in A$,
\item $\displaystyle \lim_{ n \to \omega} \|e^k_nz_k\|\ge 1/2$.
\end{enumerate} 
Now let $F_k$'s be increasing finite sets such that $\overline{\cup_{k=1}^{\infty}F_k}=A$ since $A$ is separable. Then for each $k\in \mathbb{N}$, we can choose a subsequence $n_k$'s such that 
\begin{enumerate}
\item $\displaystyle \|(\Index E)(e^k_{n_k})^{1/2}e_P  (e^k_{n_k})^{1/2}-(e^k_{n_k})\| < 1/2^k$,
\item $\displaystyle \| e^k_{n_k} z_k-z_ke^k_{n_k}\|< 1/2^k$, 
\item $\displaystyle \|e^k_{n_k}a-ae^k_{n_k}\|< 1/2^k$ \quad \text{for all $a\in F_k$},
\item $\displaystyle \|e^k_{n_k}z_k\| \ge \epsilon_k$ \quad \text{for some $\epsilon_k >1/4$}. 
\end{enumerate} 
 Then $f=[(e^k_{n_k})_k]$ is a positive contraction in $A_{\omega}\cap A'$ such that $fz\neq 0$, $fz=zf$, and $(\Index E)f^{1/2}e_P f^{1/2}=f$. Note that $f^2z \neq 0$ by (\ref{E:nonzero}).
Now we can think of a Rokhlin contraction $e\in A_{\omega}\cap A'$ such that
 $1-(\Index E)E_{\omega}(e) \lesssim f^2z$ in $A_{\omega}$. 

We apply Lemma \ref{L:technical} to $1-(\Index E)E_{\omega}(e)$, $z$, $f$ to conclude $1-(\Index E)E_{\omega}(e)\lesssim z$ in $P_{\omega}$. Now we define a map $\beta:A \to P_{\omega}$ by $\beta(a):=(\Index E)E_{\omega}(ae)$ for $a\in A$.  Let $(\Index E)E_{\omega}(e)=h \in P_{\omega}\cap P'$. Note that $\beta(p)=(\Index E)E_{\omega}(pe)=p h$ and $1-\beta(1)=1-h \lesssim z$.  It follows from \cite[Proposition 3.10]{LO} that $\beta$ is an order zero map.  Hence $\beta$ is a tracial approximate left inverse for $\iota$ with respect to $z$.  
\end{proof}

\begin{cor}\label{C:permanencebyRokhlinproperty}
Let $P \subset A$ be an inclusion of unital, separable, exact, infinite dimensional $C\sp*$-algebras of index-finite type and assume that a conditional expectation $E:A \to P$ has the generalized tracial Rohklin property.  If $A$ satisfies one of the following properties, then $P$ does too. 
\begin{enumerate}
\item simple, 
\item simple and $\mc{Z}$-absorbing provided that $P$ is nuclear,
\item simple and strict comparison property provided that $P$ is not of type I,
\item simple and stably finite.
\end{enumerate}
\end{cor}
\begin{proof}
The inclusion map $\iota: P \to A$ is tracially sequentially-split by order zero map by Theorem \ref{T:Inclusion}. Then the conclusions follows from  Theorem \ref{T:simple}, Theorem \ref{T:Z-absorbing}, Theorem \ref{T:strictcomparison}.
\end{proof}

\begin{cor}
Let $P \subset A$ be an inclusion of unital nuclear $C\sp*$-algebras of index-finite type with the generalized tracial Rokhlin property. Suppose that $A$ is a simple infinite dimensional finite $C\sp*$-algebra and $\mathcal{Z}$-absorbing.  Then $\tsr(P)=1$.
\end{cor} 
\begin{proof}
By Corollary \ref{C:permanencebyRokhlinproperty}, $P$ is also (stably) finite simple and $\mc{Z}$-absorbing. Then $\tsr(P)=1$ by \cite[Theorem 6.7]{Ro}.
\end{proof}

Though we can prove the permanence of stable rank one under an additional condition that  $\mc{Z}$-absorption, what we hope to prove is the following; 

\begin{question} 
Given $\phi:A\to B$ a tracially sequentially-split $*$-homomorphism by order zero map we assume that $B$ is a simple unital finite separable $C\sp*$-algebra of stable rank one. Then is it true that $A$ has stable rank one? 
\end{question} 

We believe that the above statement would be true if the following question is true;

\begin{question}
Let $G$ be a finite group, $A$ be a unital separable finite simple $C^*$-
algebra, and $\alpha$ be an action of $G$ on $A$ with the generalized tracial Rokhlin property
in the sense of Hirshberg and Orovitz. Assume that $\tsr(A) = 1$.
Then, is it true that $\tsr(A \rtimes_{\alpha} G) = 1$ ?
\end{question}
\section{Acknowledgements}
This research was carried out during the first author's visit to Ritsumeikan University. He would like to appreciate the department of mathematical science for its excellent support. Both authors  appreciate a fellow for the advice to use the ultrapower $C\sp*$-algebra rather than the sequence algebra. Finally, we  would like to express our sincere gratitude to a referee for his careful reading. 
%%%%%%%%%%%%%%%%%%%%%%%%%%%%%%%%%%%%%%%%%%%%%%%%%%%%%%%%% 

\end{document}